\documentclass[11pt,a4paper]{article}
\usepackage[affil-it]{authblk}
\usepackage[T1]{fontenc}
\usepackage[utf8]{inputenc}
\usepackage{lmodern}
\usepackage[margin=25mm]{geometry}
\usepackage{amsmath,amssymb,amsthm,mathtools}
\usepackage{microtype}
\usepackage[hidelinks]{hyperref}
\numberwithin{equation}{section}
\newtheorem{theorem}{Theorem}[section]
\newtheorem{proposition}[theorem]{Proposition}
\newtheorem{lemma}[theorem]{Lemma}
\theoremstyle{definition}
\newtheorem{example}[theorem]{Example}
\theoremstyle{remark}

\newcommand{\C}{\mathbb C}
\newcommand{\R}{\mathbb R}
\newcommand{\Z}{\mathbb Z}
\newcommand{\Id}{\mathrm{Id}}
\newcommand{\one}{\mathbf 1}
\newcommand{\diag}{\operatorname{diag}}
\newcommand{\rank}{\operatorname{rank}}
\newcommand{\Ran}{\operatorname{Ran}}
\newcommand{\spann}{\operatorname{span}}
\newcommand{\Cay}{\operatorname{Cay}}
\newcommand{\ip}[2]{\langle #1,#2\rangle}

\title{The fractional Calder\'on problem on finite Cayley graphs}

\author{Li Li \thanks{lili19940301@mail.tsinghua.edu.cn}}
\affil{Yau Mathematical Sciences Center, Tsinghua University, Beijing, China}
\date{}

\begin{document}
\maketitle
\vspace{-2em}

\begin{abstract}
We study a  Calder\'on-type inverse problem for the spectral fractional
Laplacian on finite Cayley graphs. For a known self-adjoint operator,
recovery from a single measurement reduces to two requirements: full rank of the
boundary (exterior)-to-interior coupling and nonvanishing of the interior solution.
We verify both requirements on Cartesian products of arbitrary finite
groups with even cycles, using consecutive cyclic halves as the
accessible and inaccessible sets. Fractionalization increases the
coupling rank, and one fixed nonnegative nonzero
input determines every nonnegative potential. Our argument relies on a discrete maximum principle and an explicit resolvent formula.
\end{abstract}

\section{Introduction}\label{sec:introduction}

The classical stationary Schr\"odinger equation is
\[
 (-\Delta+q)u=0\quad\text{in }\Omega,\qquad u|_{\partial\Omega}=f.
\]
The associated Dirichlet-to-Neumann (DN) map takes $f$ to the Neumann derivative
$\partial_\nu u$. The inverse problem asks whether this map determines
the potential $q$. The global uniqueness theorem of
Sylvester--Uhlmann \cite{SylvesterUhlmann} is the central reference for this classical Calder\'on problem.

For $0<s<1$, the fractional Schr\"odinger equation is
\begin{equation}\label{eq:continuum}
 ((-\Delta)^s+q)u=0\quad\text{in }\Omega,\qquad
 u=f\quad\text{in }\Omega_e:=\R^n\setminus\overline\Omega.
\end{equation}
The associated exterior DN measurement 
$$\Lambda_qf=((-\Delta)^su)|_{\Omega_e}.$$
Ghosh--Salo--Uhlmann \cite{GhoshSaloUhlmann} obtained the unique determination of $q$ from
partial knowledge of the DN map, using strong uniqueness and Runge approximation for the fractional Laplacian.
Ghosh--R\"uland--Salo--Uhlmann \cite{GhoshRulandSaloUhlmann} subsequently
proved uniqueness and reconstruction from one fixed nontrivial exterior
datum under their regularity assumptions. Their recovery has two steps:
first determine the solution from the exterior measurement, and then use
\[
 q=-\frac{(-\Delta)^su}{u}.
\]
Here the denominator is nonzero, which is ensured by the unique continuation property of the fractional Laplacian in this continuum case.

In this paper, we consider the corresponding finite problem with a known graph and
known edge weights. For a partition of its vertices into an inaccessible
set $I$ and an accessible set $B$, corresponding to ``interior'' and ``boundary'' (or ``exterior'') respectively in the continuum case.
Let $L$ be the graph
Laplacian and set $K=L^s$, defined by spectral calculus. For a real
potential $q$ on $I$ and a boundary input $f$ on $B$, we solve
\begin{equation}\label{eq:intro-problem}
 (Ku)(x)+q(x)u(x)=0\quad(x\in I),\qquad u|_B=f
\end{equation}
and define the DN map by
\begin{equation}\label{eq:DN-map}
  \Lambda_qf=(Ku)|_B. 
\end{equation}
Here the input and output
are boundary vectors. Our main theorem below gives an infinite class of graphs
on which a single measurement determines every nonnegative potential.

\begin{theorem}[Single-measurement recovery on product Cayley graphs]
\label{thm:positive-single}
Let $H$ be any finite group (including the trivial case) with an inverse-symmetric generating set
$T$. For $m\ge3$,
let $L$ be the unweighted Laplacian of
\[
 G=H\times\Z_{2m},\qquad
 S=(T\times\{0\})\cup\{(e_H,1),(e_H,-1)\}.
\]
We choose consecutive cyclic halves
\[
 I=H\times\{0,\ldots,m-1\},\qquad
 B=H\times\{m,\ldots,2m-1\}.
\]
For every fixed $0<s<1$,
\begin{equation}\label{eq:main-ranks}
 \rank L_{IB}=2|H|,\qquad
 \rank(L^s)_{IB}=m|H|=|I|.
\end{equation}
We fix a real nonzero input $f\ge0$ on $B$.
For every $q:I\to[0,\infty)$, the equation \eqref{eq:intro-problem} has
a unique solution with $u(x)>0$ for all $x\in I$. The single output
$g=\Lambda_qf$ uniquely determines $q$.
More precisely, with $K=L^s$ in \eqref{eq:intro-problem}, we have the reconstruction formula:
\begin{equation}\label{eq:positive-recovery}
 v=K_{BI}^{-1}(g-K_{BB}f),\qquad
 q(x)=-\frac{(K_{II}v+K_{IB}f)(x)}{v(x)}.
\end{equation}
Here $v$ is the restriction of the solution $u$ to $I$ (corresponding to $u_I$ in Theorem \ref{thm:single}).
\end{theorem}

Here $K_{BI}$ means rows in $B$ and columns in $I$. The input may be supported at one vertex, but the output is measured at all $m|H|$ boundary vertices. Thus a single measurement is a vector of
$m|H|$ values for $m|H|$ unknown potential values.

\subsection{Connection with earlier literature}

For stationary finite-graph problems, Boyer--Garzella--Guevara Vasquez
\cite{BoyerGarzellaGuevaraVasquez} study conductivities and potentials
using DN matrices and products of solutions. For evolution data,
Bl\aa sten--Isozaki--Lassas--Lu \cite{BlastenIsozakiLassasLu} relate
heat and random-walk observations to spectral recovery.
Covi--Lassas \cite{CoviLassas} study long-jump random walks whose
observable transition-matrix powers determine a gauge class under
their rank assumptions. These works motivate the question of how
observable-to-hidden coupling affects parameter recovery in inverse problems.

In the continuum, R\"uland--Salo \cite{RulandSalo} develop quantitative
approximation and logarithmic stability for the fractional Calder\'on
problem initiated in \cite{GhoshSaloUhlmann}. The anisotropic fractional Calderón problem has been studied by Feizmohammadi--Ghosh--Krupchyk--Uhlmann
\cite{FeizmohammadiGhoshKrupchykUhlmann}: local source-to-solution data
for the fractional Laplace–Beltrami operator on a closed Riemannian manifold determine
its geometry up to an isometry. Feizmohammadi--Krupchyk--Uhlmann
\cite{FeizmohammadiKrupchykUhlmann} further study fractional
Schr\"odinger Cauchy data and recover both geometry and a potential
under geometric assumptions. These geometric results use spectral
fractional powers, as does our finite model, but their data
are different from those here.

Our goal here is to establish a direct discrete analogue to the single measurement statement for the fractional Calder\'on problem \cite{GhoshRulandSaloUhlmann}.
Here the graph and weights are fixed, the unknown is a diagonal potential. Our spectral fractional $L^s$ is not the
distance-based transition operator of Covi--Lassas. The analogy with
\cite{GhoshRulandSaloUhlmann} is the recovery mechanism: first recover
the interior state, then divide to obtain the potential. Here inversion
of $K_{BI}$ replaces continuum state recovery, and a discrete maximum
principle supplies strict positivity. We remark that the unique continuation  in the continuum does not have a direct analogue in the discrete setting. We refer readers to \cite{FernandezRoncalRuland} for discussions on discrete unique
continuation. Instead, our result 
verifies a precise rank condition for specified partitions.

\subsection{Organization}

The rest of this paper is organized in the following way. Section~\ref{sec:preliminaries} provides preliminary knowledge of weighted graphs, Cayley graphs,
tensor spaces, and spectral fractional powers. We provide statements for general self-adjoint operators in Section~\ref{sec:general}. More precisely, Section~\ref{sec:general}
derives the DN matrix, the single measurement criterion and its
obstructions, and the weaker product-spanning condition for full DN
measurements. We focus on the spectral fractional Laplacian on Cayley graphs in Section~\ref{sec:fractional}. More precisely, Section~\ref{sec:fractional} proves positivity, computes the cycle
resolvent, establishes full fractional coupling rank on product graphs,
and completes the proof of Theorem~\ref{thm:positive-single}.
Section~\ref{sec:examples} gives explicit examples clarifying the dependence on the graph and partition.

\section{Preliminaries}\label{sec:preliminaries}

Let $X=I\sqcup B$ be a partition of a finite set, with inaccessible vertices $I$ and accessible vertices $B$. The words ``interior'' and ``boundary'' refer to
this chosen partition; an abstract graph has no distinguished interior or boundary.

\subsection{Functions and matrix notations}

For a finite set $X$ of size $N$, a function $u:X\to\C$ is a vector in
$\C^X\simeq\C^N$ after an ordering of $X$ is chosen. We use counting
measure and the inner product linear in its first argument,
\[
 \ip{u}{v}_X=\sum_{x\in X}u(x)\overline{v(x)}=v^*u.
\]
Here $M^*=\overline M^{\,T}$ denotes the adjoint matrix. A matrix is
self-adjoint, or Hermitian, if $M=M^*$; a real self-adjoint matrix is
symmetric. We write $M\ge0$ if $v^*Mv\ge0$ for every $v$, and $M>0$
if the inequality is strict for $v\ne0$.

We use $e_x$ to denote the vector which has value one at $x$ and zero elsewhere. For a potential
$q:I\to\R$, multiplication by $q$ is the diagonal matrix
$Q=\diag(q(x))_{x\in I}$. The Hadamard product is componentwise
multiplication:
\[
 (a\odot b)(x)=a(x)b(x).
\]
It differs from a matrix product. For subspaces $U,V\subset\C^I$,
$U\odot V$ denotes the subspace spanned by the Hadamard products of elements from $U$ and $V$. We write
$\overline V=\{\overline v:v\in V\}$.

\subsection{The graph Laplacian}

A finite undirected weighted graph consists of vertices $X$ and weights
$w_{xy}=w_{yx}\ge0$, with $w_{xx}=0$. An edge joins $x$ to $y$
exactly when $w_{xy}>0$. A path joins a sequence of distinct vertices
through edges, and the graph is connected provided that each pair of vertices can
be joined by a path. The weighted degree is $d_x=\sum_yw_{xy}$. The graph Laplacian is
\begin{equation}\label{eq:graph-L}
 (Lu)(x)=\sum_yw_{xy}(u(x)-u(y)),\qquad L=D-W,
\end{equation}
where $D=\diag(d_x)$ and $W=(w_{xy})$ is the adjacency matrix.
In an unweighted graph each edge weight is one. A graph is
$D_0$-regular provided that each vertex has degree $D_0$.

Pairing the terms indexed by $(x,y)$ and $(y,x)$ we obtain
\begin{equation}\label{eq:graph-energy}
 \ip{Lu}{u}_X=\frac12\sum_{x,y}w_{xy}|u(x)-u(y)|^2\ge0.
\end{equation}
Thus $L=L^*\ge0$. Also $L\one=0$, and connectedness gives
$\ker L=\spann\{\one\}$. We refer readers to \cite{Chung} for more background knowledge of spectral graph theory.

\subsection{Finite groups and Cayley graphs}

A group is a set with an associative operation, an identity $e$, and
an inverse $g^{-1}$ for each element $g$. We say that a group is abelian provided that $gh=hg$ for
all elements. For an abelian group it is conventional to use additive
notation, writing $g+h$ instead of $gh$. A finite group has finitely
many elements. An example is the cyclic group $\Z_M$, with addition
modulo $M$. We refer readers to
\cite{Serre} for further background on finite groups.

A subset $S\subset G$ generates $G$ if every group element is a product
of elements of $S$. The generating set $S$ is said to be inverse-symmetric provided that 
\[
 e\notin S,\qquad S=S^{-1}:=\{s^{-1}:s\in S\}.
\]
Assign weights
$a_s>0$ satisfying $a_s=a_{s^{-1}}$. The Cayley graph $\Cay(G,S)$
joins $g$ to $gs$ for every $s\in S$. Its weighted Laplacian is
\begin{equation}\label{eq:group-L}
 (Lf)(g)=\sum_{s\in S}a_s(f(g)-f(gs)).
\end{equation}
It is connected because $S$ generates $G$, and is regular with degree
$\sum_sa_s$. A generator of order two is counted once.
Left multiplication $g\mapsto hg$ preserves all edges and weights.
Thus the operator is left-translation invariant, although the partition
$I\sqcup B$ usually breaks this invariance.

\begin{example}[A cycle]\label{ex:cycle}
For the cyclic group $\Z_M$ of order $M\ge3$, and $S=\{1,-1\}$, the unweighted Cayley
graph is the cycle $C_M$. Its Laplacian is
\[
 (Lf)(j)=2f(j)-f(j+1)-f(j-1),
\]
where the indices are taken modulo $M$.
\end{example}

\subsection{Cartesian products, tensor spaces, and product groups}

\paragraph{Products of sets and tensor spaces.}
Let $X=\{x_1,\ldots,x_n\}$ and $Y=\{y_1,\ldots,y_p\}$ be finite sets with fixed orderings. For the product set $X\times Y$, we use the ordering:
\[
 (x_1,y_1),\ldots,(x_1,y_p),\;
 (x_2,y_1),\ldots,(x_2,y_p),\;\ldots,\;
 (x_n,y_1),\ldots,(x_n,y_p).
\]

A function on this set can be interpreted as a vector in the tensor space
\begin{equation}\label{eq:tensor-space}
 \C^{X\times Y}\simeq\C^X\otimes\C^Y,
 \qquad\dim(\C^X\otimes\C^Y)=|X||Y|=np.
\end{equation}
Here $e_{x_j}\otimes e_{y_k}$ corresponds to the function that equals one at
$(x_j,y_k)$ and zero elsewhere; these vectors form a basis in the
order specified above. A simple tensor $v\otimes z$ is the function
$$(x,y)\mapsto v(x)z(y),$$ 
with column-vector representation
\[
 v\otimes z=
 \begin{pmatrix}v(x_1)z\\v(x_2)z\\\vdots\\v(x_n)z\end{pmatrix}.
\]
A general function is a finite sum of simple tensors. The inner product on this space satisfies
\begin{equation}\label{eq:tensor-inner-product}
 \ip{v\otimes z}{v'\otimes z'}_{X\times Y}
 =\sum_{x,y}v(x)z(y)\overline{v'(x)z'(y)}
 =\ip{v}{v'}_X\ip{z}{z'}_Y.
\end{equation}
Therefore tensor products of orthonormal bases in the two factors form an
orthonormal basis on the product. In particular,
$\{e_{x_j}\otimes e_{y_k}\}_{j,k}$ is an orthonormal basis.

Let $A$ act on $\C^X$ and let $D$ act on
$\C^Y$. the Kronecker product is defined by
\[
 (A\otimes D)(v\otimes z)=(Av)\otimes(Dz).
\]
More explicitly, in the chosen product ordering it is
the $np\times np$ matrix
\begin{equation}\label{eq:Kronecker-product-matrix}
 A\otimes D=
 \begin{pmatrix}
  a_{11}D&a_{12}D&\cdots&a_{1n}D\\
  a_{21}D&a_{22}D&\cdots&a_{2n}D\\
  \vdots&\vdots&\ddots&\vdots\\
  a_{n1}D&a_{n2}D&\cdots&a_{nn}D
 \end{pmatrix},
\end{equation}
where each displayed block is $p\times p$. We will use the following basic identities:
\[
 (A\otimes D)(E\otimes F)=(AE)\otimes(DF),
 \qquad
 (A\otimes D)^*=A^*\otimes D^*.
\]
The product formula holds for general  rectangular matrices. In particular, when $A$ and $D$ are invertible square matrices,
\[
 (A\otimes D)^{-1}=A^{-1}\otimes D^{-1}.
\]

\paragraph{Cartesian products of graphs.}
The Cartesian product of graphs with vertex sets $X,Y$ joins $(x,y)$
to $(x',y)$ when $x$ and $x'$ are adjacent, and to $(x,y')$ when
$y$ and $y'$ are adjacent. Only one coordinate changes along an edge.
If their Laplacians are $L_1,L_2$, the product Laplacian is
\begin{equation}\label{eq:cartesian-L}
 L=L_1\otimes\Id_Y+\Id_X\otimes L_2.
\end{equation}
It acts on the tensor space in \eqref{eq:tensor-space}. Indeed, for a
simple tensor,
\[
 L(v\otimes z)=(L_1v)\otimes z+v\otimes(L_2z),
\]
and simple tensors span the whole tensor space. Equivalently, for a general $u$,
\[
 (Lu)(x,y)
 =\sum_{x'}w^{(1)}_{xx'}(u(x,y)-u(x',y))+\sum_{y'}w^{(2)}_{yy'}(u(x,y)-u(x,y')).
\]

\paragraph{Direct products of groups and their Cayley graphs.}
For two groups $G_1, G_2$, the direct product $G_1\times G_2$ is a group equipped with the componentwise
multiplication:
$$(g_1,g_2)(h_1,h_2)=(g_1h_1,g_2h_2).$$
If inverse-symmetric $S_1,S_2$ generate $G_1, G_2$ respectively, then we choose
\[
 S=(S_1\times\{e_2\})\cup(\{e_1\}\times S_2),
\]
which generates $G_1\times G_2$.
$\Cay(G_1\times G_2,S)$ is called the Cartesian product of the two Cayley
graphs $\Cay(G_1,S_1), 
\Cay(G_2,S_2)$. In Theorem~\ref{thm:positive-single}, we focus on
$H\times\Z_M$ with generators $(t,0)$, $t\in T$, and $(e_H,\pm1)$;
for this choice \eqref{eq:cartesian-L} applies.

\subsection{Spectral powers and the heat semigroup}

Let $L$ be the Laplacian on a connected undirected weighted graph consisting of $N$ vertices $X$.
The spectral theorem gives an orthonormal eigenbasis $\phi_j$:
\[
 L\phi_j=\lambda_j\phi_j,\qquad
 0=\lambda_0<\lambda_1\le\cdots\le\lambda_{N-1}.
\]
For $0<s<1$, we define
\begin{equation}\label{eq:spectral-power}
 L^su=\sum_{j=0}^{N-1}\lambda_j^s\ip{u}{\phi_j}_X\phi_j.
\end{equation}
Thus $L^s$ is self-adjoint, nonnegative, and has the same constant kernel
as $L$. Likewise
\[
 e^{-tL}u=\sum_{j=0}^{N-1}e^{-t\lambda_j}\ip{u}{\phi_j}_X\phi_j
\]
gives the solution operator of the graph heat equation 
\[\partial_tu+Lu=0.\]
Derivation of the integral formulas and strict sign properties needed for inverse recovery will be presented later in Section~\ref{sec:fractional}. We refer readers to \cite{BianchiEtAl} for more discussions on spectral fractional graph operators and related random walks.

\paragraph{Spectral powers on a Cartesian product.}
Now let $|X|=n$ and $|Y|=p$. Suppose
\[
 L_1\psi_\alpha=\mu_\alpha\psi_\alpha\quad(1\le\alpha\le n),
 \qquad
 L_2\phi_k=\lambda_k\phi_k\quad(1\le k\le p),
\]
where both eigenbases are orthonormal.
For $L=L_1\otimes\Id_Y+\Id_X\otimes L_2$, the tensor rule gives
\[
 L(\psi_\alpha\otimes\phi_k)
 =(\mu_\alpha+\lambda_k)(\psi_\alpha\otimes\phi_k).
\]
By \eqref{eq:tensor-inner-product}, these $|X||Y|$ eigenvectors form
an orthonormal basis. Hence the fractional product Laplacian is
\begin{equation}\label{eq:fractional-product-expansion}
 L^su=\sum_{\alpha,k}(\mu_\alpha+\lambda_k)^s
 \ip{u}{\psi_\alpha\otimes\phi_k}_{X\times Y}
 (\psi_\alpha\otimes\phi_k).
\end{equation}
In particular, 
\[
 L^s(\psi_\alpha\otimes\phi_k)
 =(\mu_\alpha+\lambda_k)^s(\psi_\alpha\otimes\phi_k).
\]
Note that
$$(\mu_\alpha+\lambda_k)^s\ne\mu_\alpha^s+\lambda_k^s$$
in general, so $L^s$
is not the same as $(L^s_1\otimes\Id_Y+\Id_X\otimes L^s_2)$.

There is another useful way to express $L^s$. Let
$\Pi_\alpha=\psi_\alpha\psi_\alpha^*$ be the rank-one orthogonal
projection from $\C^X$ onto the subspace spanned by $\psi_\alpha$. Then
\begin{equation}\label{eq:fractional-product-decomposition}
 L^s=\sum_\alpha\Pi_\alpha\otimes(\mu_\alpha\Id_Y+L_2)^s.
\end{equation}
Indeed, acting on the single tensor $\psi_\alpha\otimes \phi_k$, the right hand side of \eqref{eq:fractional-product-decomposition} becomes
$$\psi_\alpha\otimes (\mu_\alpha\Id_Y+L_2)^s\phi_k=
\psi_\alpha\otimes (\mu_\alpha+\lambda_k)^s\phi_k
= (\mu_\alpha+\lambda_k)^s(\psi_\alpha\otimes\phi_k)= L^s(\psi_\alpha\otimes\phi_k).$$

This decomposition also explains the cross-block calculation used
later. Let $Y=I_Y\sqcup B_Y$. We choose
$I=X\times I_Y$, $B=X\times B_Y$. From \eqref{eq:Kronecker-product-matrix}, Restricting \eqref{eq:fractional-product-decomposition} gives
\begin{equation}\label{eq:product-cross-decomposition}
 (L^s)_{IB}=\sum_\alpha\Pi_\alpha\otimes
 \big[(\mu_\alpha\Id_Y+L_2)^s\big]_{I_YB_Y}.
\end{equation}
Here $I_YB_Y$ means rows in $I_Y$ and columns in $B_Y$.

Let $V=(\psi_1,\ldots,\psi_n)$, so $V^*V=\Id_X$.
If $E_{\alpha\alpha}$ is the $n\times n$ matrix with its only
nonzero entry equal to one in position $(\alpha,\alpha)$, then
$$V^*\Pi_\alpha V=E_{\alpha\alpha}.$$ The Kronecker product formula
therefore gives
\begin{align*}
 &(V^*\otimes\Id_{I_Y})(\Pi_\alpha\otimes \big[(\mu_\alpha\Id_Y+L_2)^s\big]_{I_YB_Y})
        (V\otimes\Id_{B_Y})\\
 &\hspace{2em}=(V^*\Pi_\alpha V)\otimes \big[(\mu_\alpha\Id_Y+L_2)^s\big]_{I_YB_Y}
 =E_{\alpha\alpha}\otimes \big[(\mu_\alpha\Id_Y+L_2)^s\big]_{I_YB_Y}.
\end{align*}
Summing over $\alpha$ yields
\begin{equation}\label{eq:unitary-cross-blocks}
 (V^*\otimes\Id_{I_Y})(L^s)_{IB}(V\otimes\Id_{B_Y})
 =\bigoplus_{\alpha=1}^n \big[(\mu_\alpha\Id_Y+L_2)^s\big]_{I_YB_Y}.
\end{equation}
The block matrix on the right hand side is explicitly
\[
 \bigoplus_{\alpha=1}^n \big[(\mu_\alpha\Id_Y+L_2)^s\big]_{I_YB_Y}=
 \begin{pmatrix}
  \big[(\mu_1\Id_Y+L_2)^s\big]_{I_YB_Y}&\cdots&0\\
  \vdots&\ddots&\vdots\\
  
  0&\cdots&\big[(\mu_n\Id_Y+L_2)^s\big]_{I_YB_Y}
 \end{pmatrix}.
\]

Unitary changes of basis preserve rank, and
the rank of a block diagonal matrix is the sum of its block ranks.  Consequently,
\begin{equation}\label{eq:rank-cross-blocks}
 \rank(L^s)_{IB}=\sum_{\alpha=1}^n\rank \big[(\mu_\alpha\Id_Y+L_2)^s\big]_{I_YB_Y}.
\end{equation}
For $X=H$ and $Y=\Z_{2m}$, these are the main ingredients used in the proof of Theorem~\ref{thm:rank-enhance} later.

\section{The single-measurement statement for general self-adjoint operators}\label{sec:general}

\subsection{Dirichlet problems and the DN matrix}

Let $K$ be any Hermitian matrix acting on $\C^X$, with a nonempty proper partition
$X=B\sqcup I$. With boundary coordinates first, we write
\begin{equation}\label{eq:K-blocks}
 K=\begin{pmatrix}K_{BB}&K_{BI}\\K_{IB}&K_{II}\end{pmatrix},
 \qquad K_{BI}=K_{IB}^*.
\end{equation}
The notation $K_{IB}$ means rows in $I$ and columns in $B$, which is
a rectangular off-diagonal block. For a real potential $q:I\to\R$, set
\begin{equation}\label{eq:Aq}
 Q=\diag(q),\qquad A_q=K_{II}+Q.
\end{equation}
We call $q$ \emph{admissible} when $A_q$ is invertible. Given
$f\in\C^B$, the Dirichlet problem is
\begin{equation}\label{eq:Dirichlet}
 (Ku)_I+Qu_I=0,\qquad u_B=f.
\end{equation}
In blocks this is 
$$A_qu_I=-K_{IB}f.$$ 
For admissible $q$, the interior
Poisson operator is
\begin{equation}\label{eq:Poisson}
 u_I=P_qf,\qquad P_q=-A_q^{-1}K_{IB}.
\end{equation}
The potential acts only inside $I$. To formulate the inverse problem, we define the DN measurement
$$\Lambda_qf=(Ku)_B.$$
Substituting \eqref{eq:Poisson} gives
\begin{equation}\label{eq:DN}
\Lambda_q=K_{BB}-K_{BI}A_q^{-1}K_{IB}.
\end{equation}
This Schur-complement formula eliminates the interior variables and
gives the matrix relation between boundary input and output.
Since $A_q=A_q^*$, the DN matrix is self-adjoint.

\subsection{Reconstruction from a single measurement}

We fix $f\in\C^B$. Let $W\subset B$ be the observation set.
We write $g=R_W\Lambda_qf$, where $R_W$ is the restriction from $B$ to $W$.
From \eqref{eq:Poisson} and \eqref{eq:DN}, the measurement equation is
\begin{equation}\label{eq:measurement}
 g=K_{WB}f+K_{WI}u_I.
\end{equation}

\begin{theorem}[Single-measurement reconstruction]\label{thm:single}
Given a Hermitian $K$, assume $q$ is real and admissible, and
\begin{equation}\label{eq:full-rank}
 \rank K_{IW}=|I|.
\end{equation}
For the fixed input $f$, also assume that
\begin{equation}\label{eq:nonzero-state}
 u_I(x)\ne0\quad\text{for each }x\in I.
\end{equation}
Then $g=R_W\Lambda_qf$ uniquely determines $q$. More explicitly, set
\[
 J=(K_{WI}^*K_{WI})^{-1}K_{WI}^*.
\]
Then
\begin{equation}\label{eq:single-formula}
 u_I=J(g-K_{WB}f),\qquad
 q(x)=-\frac{(K_{II}u_I+K_{IB}f)(x)}{u_I(x)}.
\end{equation}
\end{theorem}
\begin{proof}
Adjoint symmetry makes \eqref{eq:full-rank} equivalent to full column
rank of $K_{WI}$. Hence $J$ is well defined and $JK_{WI}=\Id_I$.
Apply $J$ to \eqref{eq:measurement} to recover $u_I$. The interior
equation in \eqref{eq:Dirichlet} gives the quotient in \eqref{eq:single-formula}, and
\eqref{eq:nonzero-state} permits division.
\end{proof}

This proof separates the two requirements: the block rank recovers $u_I$, and the non-vanishing
condition ensures the recovery of each potential value. Full rank \eqref{eq:full-rank} implies that $|W|\ge|I|$ is necessary for this mechanism.  In the real symmetric setting, full rank \eqref{eq:full-rank} is necessary for local recovery of
potentials from one input, as the next proposition shows.

\begin{proposition}\label{prop:rank-obstruction}
Suppose $K$ and $f$ are real, $q$ is real admissible, and every entry of
$u_I$ is nonzero. If 
$$\rank K_{WI}<|I|,$$ 
there exists a nonconstant curve
of real admissible potentials near $q$ giving exactly the same measurement
$R_W\Lambda_qf$.
\end{proposition}
\begin{proof}
Choose $0\ne w\in\ker K_{WI}$ and define 
$$v_\varepsilon=u_I+\varepsilon w.$$
For sufficiently small real $\varepsilon$, its entries are nonzero. Define
\[
 q_\varepsilon(x)
 =-\frac{(K_{II}v_\varepsilon+K_{IB}f)(x)}{v_\varepsilon(x)}.
\]
Then $(f,v_\varepsilon)$ solves the Dirichlet problem \eqref{eq:intro-problem} with potential
$q_\varepsilon$ instead of $q$. Its output is unchanged because $K_{WI}w=0$.
Also $q_0=q$, and $q_\varepsilon$ is admissible for sufficiently small $\varepsilon$
by continuity of the determinant.
To compute the derivative with respect to $\varepsilon$, we write
$$n_x(\varepsilon)=(K_{II}v_\varepsilon+K_{IB}f)(x).$$
The original interior equation gives $$n_x(0)=-q(x)u_I(x).$$
Also note that
$$n_x'(0)=(K_{II}w)(x).$$ 
The quotient rule yields
\begin{align*}
 \left.\frac{d}{d\varepsilon}\right\vert_{\varepsilon= 0}q_\varepsilon
 (x)
 &=-\frac{(K_{II}w)(x)u_I(x)-n_x(0)w(x)}{u_I(x)^2}\\
 &=-\frac{(K_{II}w)(x)+q(x)w(x)}{u_I(x)}
 =-\frac{(A_qw)(x)}{u_I(x)}.
\end{align*}
Therefore
\[
 \left.\frac{d}{d\varepsilon}\right\vert_{\varepsilon= 0}q_\varepsilon=-\diag(u_I)^{-1}A_qw\ne0
\]
since both $A_q$ and $\diag(u_I)$ are invertible and $w\ne0$.
This proves that the curve is nonconstant.
\end{proof}

\subsection{Full DN data and the weaker product-spanning criterion}

We also remark that if we have the full knowledge of $R_{W_N}\Lambda_q$ instead of $R_{W_N}\Lambda_qf$ for a single input $f$, then it is possible to obtain the unique determination of $q$ with much smaller $W$. 

The following discrete Alessandrini identity is similar to its counterpart in \cite{BoyerGarzellaGuevaraVasquez}. We include its derivation for self-completeness.

\begin{proposition}\label{prop:Alessandrini}
Let $q_1,q_2$ be real admissible potentials. Set $h=q_1-q_2$ and
$H_q=\diag(h)$. For boundary data $f\in\C^B$ and $E_Wg$ with
$g\in\C^W$, let $u_1,u_2$ be the corresponding solutions for $q_1,q_2$.
Then
\begin{align}
 \Lambda_{q_1}-\Lambda_{q_2}
 &=P_{q_2}^*H_qP_{q_1},\label{eq:DN-difference}\\
 \ip{(\Lambda_{q_1}-\Lambda_{q_2})f}{E_Wg}_B
 &=\sum_{x\in I}h(x)u_1(x)\overline{u_2(x)}.\label{eq:Hadamard}
\end{align}
Here $E_W$ extends an input by zero from $W$ to $B$.
\end{proposition}
\begin{proof}
Based on \eqref{eq:DN} and the matrix identity
\[
 A_{q_2}^{-1}-A_{q_1}^{-1}
 =A_{q_2}^{-1}(A_{q_1}-A_{q_2})A_{q_1}^{-1}
 =A_{q_2}^{-1}H_qA_{q_1}^{-1},
\]
we obtain \eqref{eq:DN-difference}. Pairing with $E_Wg$ gives
\[
 (E_Wg)^*P_{q_2}^*H_qP_{q_1}f
 =(u_2|_I)^*H_q(u_1|_I),
\]
which is exactly the finite sum in \eqref{eq:Hadamard}.
\end{proof}

We define
\begin{equation}\label{eq:partial-Poisson}
 P_{q,W}=-A_q^{-1}K_{IW}.
\end{equation}
Invertibility of $A_q$ implies
\begin{equation}\label{eq:reachable-dim}
 \dim\Ran(P_q)=\rank K_{IB},\qquad
 \dim\Ran(P_{q,W})=\rank K_{IW}.
\end{equation}
If $R_W\Lambda_{q_1}=R_W\Lambda_{q_2}$ and
\begin{equation}\label{eq:product-spanning}
 \Ran(P_{q_1})\odot\overline{\Ran(P_{q_2,W})}=\C^I,
\end{equation}
then \eqref{eq:Hadamard} implies $h=0$, so $q_1=q_2$.

Note that \eqref{eq:product-spanning} necessarily requires
\begin{equation}\label{eq:product-dimension}
 |I|\le(\rank K_{IB})(\rank K_{IW})\le|B||W|.
\end{equation}
Thus products may span the interior even when individual solution spaces
do not. This is a discrete counterpart of the density argument in
Calder\'on-type inverse problems, and suggests a smaller observation set than
the one-input mechanism requires, although the dimension bound \eqref{eq:product-dimension} alone is not a sufficient spanning condition.

\section{Fractional Laplacian on Cayley graphs}\label{sec:fractional}

\subsection{Integral formulas and nonlocal interaction}

We first establish the sign properties of the spectral fractional Laplacian on any finite connected weighted
graph. No group structure is needed for these properties.

\begin{lemma}[Strict positivity of the heat kernel]\label{lem:heat}
For every $t>0$, all entries of $e^{-tL}$ are strictly positive.
Moreover, $e^{-tL}\one=\one$.
\end{lemma}
\begin{proof}
Choose $\rho\ge\max_xd_x$ and set $M=\rho\Id-L$.
Its diagonal entries are $\rho-d_x\ge0$, and its off-diagonal entries
are edge weights $w_{xy}\ge0$. Thus
\[
 e^{-tL}=e^{-\rho t}e^{tM}
 =e^{-\rho t}\sum_{k=0}^\infty\frac{t^k}{k!}M^k
\]
is a sum of entrywise nonnegative matrices.
For $x\ne y$, connectedness enables us to choose a path
$x=x_0,x_1,\ldots,x_\ell=y$ with positive edge weights. Matrix
multiplication gives, for $\ell\ge2$,
\[
 (M^\ell)_{xy}
 =\sum_{z_1,\ldots,z_{\ell-1}\in X}
 M_{xz_1}M_{z_1z_2}\cdots M_{z_{\ell-1}y}
 \ge\prod_{j=0}^{\ell-1}w_{x_jx_{j+1}}>0.
\]
For $\ell=1$, this is simply $M_{xy}=w_{xy}>0$.
The $k=\ell$ term therefore makes $(e^{-tL})_{xy}>0$.
Diagonal entries are positive already from the $k=0$ term.
Finally, $L\one=0$ implies $e^{-tL}\one=\one$.
\end{proof}

For $\lambda\ge0$ and $0<s<1$, the scalar integral formulas are
\begin{align}
 \lambda^s
 &=c_s\int_0^\infty(1-e^{-t\lambda})\frac{dt}{t^{1+s}},
 &c_s&=\frac{s}{\Gamma(1-s)},\label{eq:scalar-heat}\\
 \lambda^s
 &=b_s\int_0^\infty t^{s-1}\frac{\lambda}{\lambda+t}\,dt,
 &b_s&=\frac{\sin(\pi s)}{\pi}.\label{eq:scalar-resolvent}
\end{align}
Applying these two identities to each eigenvalue yields
\begin{align}
 L^s&=c_s\int_0^\infty(\Id-e^{-tL})\frac{dt}{t^{1+s}},
 \label{eq:matrix-heat}\\
 L^s&=b_s\int_0^\infty t^{s-1}L(L+t\Id)^{-1}\,dt,
 \label{eq:matrix-resolvent}
\end{align}
which hold more generally for any self-adjoint positive semi-definite matrix, rather than only for the graph Laplacian.

We mention that both integrals converge in matrix norm. In the heat formula, the integrand is $O(t^{-s})$ near
zero and $O(t^{-1-s})$ at infinity.
In the resolvent formula,
$\|L(L+t\Id)^{-1}\|\le1$ near zero and it is $O(t^{-1})$
at infinity, so the integrand is $O(t^{s-1})$ near
zero and $O(t^{s-2})$ at infinity.
We also remark that the analogues of \eqref{eq:matrix-heat} and \eqref{eq:matrix-resolvent} in the continuum are the Bochner heat-semigroup
and Balakrishnan resolvent formulas. We refer readers to Kwa\'snicki
\cite[Definitions 2.3 and 2.4]{Kwasnicki} for discussions on their equivalence with other definitions.

\begin{proposition}[Positive fractional jump weights]\label{prop:jump}
For distinct vertices $x,y$, $(L^s)_{xy}<0$. We define
\begin{equation}\label{eq:jump-weight}
 w_s(x,y)=-(L^s)_{xy}
 =c_s\int_0^\infty(e^{-tL})_{xy}\frac{dt}{t^{1+s}}>0,
\end{equation}
and $w_s(x,x)=0$. Then we have
\begin{align}
 (L^su)(x)&=\sum_{y\ne x}w_s(x,y)(u(x)-u(y)),\label{eq:jump-form}\\
 \ip{L^su}{u}_X&=\frac12\sum_{x,y}w_s(x,y)|u(x)-u(y)|^2.
 \label{eq:fractional-energy}
\end{align}
\end{proposition}
\begin{proof}
The off-diagonal entries of $\Id$ vanish, so \eqref{eq:jump-weight} follows from \eqref{eq:matrix-heat}
and Lemma~\ref{lem:heat}.
Since $L^s\one=0$, its diagonal entry equals
$\sum_{y\ne x}w_s(x,y)$. This proves \eqref{eq:jump-form};
symmetrizing gives \eqref{eq:fractional-energy}.
\end{proof}

Thus $L^s$ is the Laplacian of a weighted complete graph on the same
vertices. Its weights retain information on the original graph,
but are generated by spectral calculus and are not an exact
power of graph distance in general.

\subsection{Well-posedness and a maximum principle}

\begin{proposition}\label{prop:wellposed}
For any nonempty proper partition $X=B\sqcup I$, the principal block
$(L^s)_{II}$ is positive definite. Consequently
$$A_q=(L^s)_{II}+Q>0$$ for every $q\ge0$, and the Dirichlet problem \eqref{eq:Dirichlet} with $K=L^s$ is well-posed.
\end{proposition}
\begin{proof}
Let $\widetilde v$ be the extension of $v\in\C^I$ by zero on $B$.
Then $v^*(L^s)_{II}v=\widetilde v^*L^s\widetilde v\ge0$.
If this is exactly zero, then
$\widetilde v\in\ker L^s=\spann\{\one\}$.
A constant vanishing on the nonempty set $B$ has to be zero. Therefore $v=0$.
Adding a nonnegative diagonal matrix preserves positive definiteness.
\end{proof}

Consider the inhomogeneous problem
\begin{equation}\label{eq:inhomogeneous}
 (L^su)(x)+q(x)u(x)=F(x)\quad(x\in I),\qquad u|_B=f,
\end{equation}
where $q\ge0$, $F\ge0$, and $f\ge0$.

\begin{theorem}[Discrete fractional maximum principle]
\label{thm:maximum}
(i)\, The unique solution of \eqref{eq:inhomogeneous} satisfies $u\ge0$ on
$X$. 
(ii)\, Moreover, if $f$ or $F$ is not identically zero, then
\begin{equation}\label{eq:strict-positive}
 u(x)>0\quad\text{for every }x\in I.
\end{equation}
\end{theorem}
\begin{proof}
(i)\, Since $X$ is finite, a minimum is attained. If it were negative,
an interior point $x_0$ would attain it, since $f\ge0$.
For every $y$, $u(x_0)-u(y)\le0$. Any boundary vertex $b$ satisfies
$u(b)\ge0>u(x_0)$, and $w_s(x_0,b)>0$. Hence
\[
 (L^su)(x_0)=\sum_{y\ne x_0}w_s(x_0,y)(u(x_0)-u(y))<0.
\]
The term $q(x_0)u(x_0)$ is nonpositive. The sum contradicts
$F(x_0)\ge0$. Thus $u\ge0$.

(ii)\, If $u(x_0)=0$ at an interior vertex, its equation is
\[
 F(x_0)=-\sum_{y\ne x_0}w_s(x_0,y)u(y)\le0.
\]
Since $F(x_0)\ge0$, this sum has to be zero. Every weight is positive and
every value is nonnegative, so $u(y)=0$ for all $y\ne x_0$.
Thus $u\equiv0$, forcing both $f=0$ and $F=0$, contradicting the assumption.
\end{proof}

In particular, for $F=0$, any fixed nonnegative nonzero boundary input, possibly supported at a single accessible vertex, produces
positive values at every interior vertex for every $q\ge0$. Taking $f=0$ and $F=e_x$ also shows that every entry of $A_q^{-1}$ is positive.

\subsection{Rank enhancement on product Cayley graphs}\label{subsec:rank}

To complete the proof of the main theorem, we proceed from the derivation of an explicit cycle resolvent. Pass to 
the shifted fractional cycle via the integral formula \eqref{eq:matrix-resolvent}.
Then we use the tensor decomposition \eqref{eq:product-cross-decomposition} from Section~\ref{sec:preliminaries} to include the factor $H$.

\paragraph{The cycle resolvent.}

Let $N=2m$, $m\ge3$, and let $L_C=L_{C_N}$ be the degree-two
cycle Laplacian from Example~\ref{ex:cycle}. For $a>0$, we define
\begin{equation}\label{eq:r-definition}
 \Delta(a)=\sqrt{a(a+4)},\qquad
 r(a)=\frac{a+2-\Delta(a)}2=\frac{2}{a+2+\Delta(a)}.
\end{equation}
Note that $r=r(a)$ is the smaller root of
$$r^2-(a+2)r+1=0.$$ 
Thus $0<r<1$ and
\begin{equation}\label{eq:r-identities}
 r+r^{-1}=a+2,\qquad r^{-1}-r=\Delta(a).
\end{equation}

The cycle Green-function calculations below are standard. We include computational details here for self-completeness. We refer readers to
\cite{Ellis} for more discussions on discrete Green functions.

Put
$G_a(n)=r^{|n|}/\Delta(a)$. For $n\ne0$, the root equation gives
\[
 (a+2)G_a(n)-G_a(n-1)-G_a(n+1)=0.
\]
At $n=0$, the same expression is
\[
 \frac{a+2-2r}{\Delta(a)}
 =\frac{r^{-1}-r}{\Delta(a)}=1.
\]
Hence $G_a$ is a decaying fundamental solution. Periodizing it over the cycle, we obtain
\[
 R_a(d)=\sum_{\ell\in\Z}G_a(d+\ell N).
\]
The series converges absolutely. More explicitly, for $0\le d\le N-1$,
\[
 R_a(d)
 =\frac1{\Delta(a)}
   \left(\sum_{\ell=0}^\infty r^{d+\ell N}
        +\sum_{k=1}^\infty r^{kN-d}\right)
 =\frac{r^d+r^{N-d}}{\Delta(a)(1-r^N)}.
\]

We claim
\begin{equation}\label{eq:cycle-resolvent}
 [(a\Id+L_C)^{-1}]_{ij}=R_a((i-j)\bmod N).
\end{equation}
Indeed, the fundamental-solution identity is
\begin{equation}\label{eq:Rae0}
 (a+2)R_a(d)-R_a(d-1)-R_a(d+1)=\delta_{d,0},
 \qquad (a\Id+L_C)R_a=e_0,
\end{equation}
with all indices taken modulo $N$. Let
$v^{(j)}(i)=R_a((i-j)\bmod N)$. Then
\[
 [(a\Id+L_C)v^{(j)}](i)
 =\delta_{i,j},
\]
which implies $v^{(j)}=(a\Id+L_C)^{-1}e_j$.
This is the column $j$ of its inverse, proving \eqref{eq:cycle-resolvent}.

\paragraph{Full cross-block rank for a shifted fractional cycle.}

Let $\mu\ge0$ and define
\begin{equation}\label{eq:shifted-cycle}
 F_\mu=(\mu\Id+L_C)^s,\qquad
 I_C=\{0,\ldots,m-1\},\quad B_C=\{m,\ldots,2m-1\}.
\end{equation}
Let $C_\mu=(F_\mu)_{I_CB_C}$, with both sets initially in increasing
order. Let $J_m$ reverse $m$ coordinates. Then
\begin{equation}\label{eq:reflection}
 (C_\mu J_m)_{ij}=(F_\mu)_{i,,2m-1-j},\qquad 0\le i,j\le m-1.
\end{equation}
Note that the difference $i-(2m-1-j)$ is $i+j+1$ modulo $2m$.

\begin{theorem}\label{thm:reflected}
For every $\mu\ge0$ and $0<s<1$, the real symmetric matrix $-C_\mu J_m$ is positive definite. In particular, $C_\mu$ is invertible.
\end{theorem}
\begin{proof}
Applying \eqref{eq:matrix-resolvent} to $\mu\Id+L_C$ and using the identity
\[
 (\mu\Id+L_C)(\mu\Id+L_C+t\Id)^{-1}
 =\Id-t(\mu\Id+L_C+t\Id)^{-1},
\]
for distinct $x,y$ we obtain
\begin{equation}\label{eq:shifted-offdiagonal}
 (F_\mu)_{xy}
 =-b_s\int_0^\infty t^s[(\mu+t)\Id+L_C]^{-1}_{xy}\,dt.
\end{equation}
Let $a=\mu+t$ and $r=r(a)$. Formula \eqref{eq:cycle-resolvent} gives
\begin{equation}\label{eq:reflected-entry}
 -(C_\mu J_m)_{ij}
 =b_s\int_0^\infty
 \frac{t^s(r^{i+j+1}+r^{2m-i-j-1})}
 {\Delta(a)(1-r^{2m})}\,dt.
\end{equation}
This is symmetric in $i,j$. For $z\in\C^m$, finite summation gives
\begin{equation}\label{eq:positive-integral}
 -z^*C_\mu J_mz
 =b_s\int_0^\infty\frac{t^s}{\Delta(a)(1-r^{2m})}
 \left(r\left|\sum_{i=0}^{m-1}z_ir^i\right|^2
 +r^{2m-1}\left|\sum_{i=0}^{m-1}z_ir^{-i}\right|^2\right)dt.
\end{equation}
Note that every factor outside the squares is positive. If $z\ne0$, the
polynomial $p(r)=\sum_i z_ir^i$ is nonzero and has finitely many roots.
As $t$ varies, $r(\mu+t)$ traverses an interval in $(0,1)$.
Thus $|p(r)|$ is positive on a set of positive length, and the
integral \eqref{eq:positive-integral} is strictly positive.
\end{proof}

We remark that \eqref{eq:cycle-resolvent} implies
\begin{equation}\label{eq:resolvent-crossblock}
[(\mu\Id+L_C+t\Id)^{-1}]_{I_CB_C}J_m
= \frac{r}{\Delta(a)(1-r^{2m})}(vv^T+ ww^T),
\end{equation}
where 
$$v=(1,r,\ldots,r^{m-1})^T,\qquad 
w= (r^{m-1},\ldots,r,1)^T.$$
Both $vv^T$ and $ww^T$ have rank one for each fixed $t$ so the cross block \eqref{eq:resolvent-crossblock} has rank at most two. In \eqref{eq:reflected-entry} $r$ is a variable of $t$, and this integration leads to the rank enhancement.

\paragraph{Arbitrary finite groups multiplied by even cycles.}

Let $H$ be any finite group, $h=|H|$, with an inverse-symmetric
generating set $T$ and corresponding unweighted Laplacian $L_H$. There is no restriction on whether $H$ is abelian. The trivial group is allowed, with $T$ empty and $L_H=0$. Let
\begin{equation}\label{eq:product-group}
 G=H\times\Z_{2m},\quad m\ge3,\qquad
 S=(T\times\{0\})\cup\{(e_H,1),(e_H,-1)\}.
\end{equation}
The product Laplacian \eqref{eq:cartesian-L} is 
\begin{equation}\label{eq:product-L}
 L=L_H\otimes\Id_{2m}+\Id_h\otimes L_C.
\end{equation}
We choose
\begin{equation}\label{eq:product-partition}
 I=H\times I_C,\qquad B=H\times B_C.
\end{equation}
Both sets have $mh$ vertices.

\begin{theorem}[Fractional rank enhancement on product groups]\label{thm:rank-enhance}
For \eqref{eq:product-group}--\eqref{eq:product-partition},
\begin{equation}\label{eq:rank-enhance}
 \rank L_{IB}=2h,\qquad \rank(L^s)_{IB}=mh=|I|\quad(0<s<1).
\end{equation}
\end{theorem}
\begin{proof}
We treat the classical and fractional cases separately. Recall that $m\ge3$.

\noindent\textit{The classical cross block.}
The first summand in \eqref{eq:product-L} has no contributions coupling $I$ to
$B$. Consequently,
\[
 L_{IB}=\Id_h\otimes E,\qquad E=(L_C)_{I_CB_C}.
\]
We write $I_C=(0,\ldots,m-1)$ and $B_C=(m,\ldots,2m-1)$ in
increasing order as before. Only the cycle edges $0\leftrightarrow 2m-1$
and $m-1\leftrightarrow m$ cross the partition. Thus
\[
 E=\begin{pmatrix}
    0&0_{1\times(m-2)}&-1\\
    0_{(m-2)\times1}&0_{(m-2)\times(m-2)}&0_{(m-2)\times1}\\
    -1&0_{1\times(m-2)}&0
   \end{pmatrix}.
\]
Therefore $\rank E=2$. $L_{IB}$ has $h$ diagonal blocks so
$\rank L_{IB}=2h$.

\noindent\textit{The fractional cross block.}
Choose an orthonormal eigenbasis
$\psi_1,\ldots,\psi_h$ of $L_H$, with
$$L_H\psi_\alpha=\mu_\alpha\psi_\alpha$$and $\mu_\alpha\ge0$.
Let $V_H=(\psi_1,\ldots,\psi_h)$, which is unitary. Applying
\eqref{eq:unitary-cross-blocks} with $X=H$, $Y=\Z_{2m}$,
$I_Y=I_C$, and $B_Y=B_C$, we obtain
\[
 (V_H^*\otimes\Id_m)(L^s)_{IB}(V_H\otimes\Id_m)
 =\begin{pmatrix}
   C_{\mu_1}&0&\cdots&0\\
   0&C_{\mu_2}&\cdots&0\\
   \vdots&\vdots&\ddots&\vdots\\
   0&0&\cdots&C_{\mu_h}
  \end{pmatrix},
\]
\[C_{\mu_\alpha}=\big[(\mu_\alpha\Id_{2m}+L_C)^s\big]_{I_CB_C}.
\]
By \eqref{eq:rank-cross-blocks} and 
Theorem~\ref{thm:reflected}, we conclude that
\[
 \rank(L^s)_{IB}
 =\sum_{\alpha=1}^h\rank C_{\mu_\alpha}
 =mh=|I|.
\]
\end{proof}

\medskip
\begin{proof}[Completion of the proof of Theorem~\ref{thm:positive-single}]
Proposition~\ref{prop:wellposed} shows that every $q\ge0$ is admissible.
For the fixed nonnegative nonzero input, Theorem~\ref{thm:maximum}
gives $u_I(x)>0$ at every interior vertex. Theorem~\ref{thm:rank-enhance}
establishes \eqref{eq:main-ranks} and makes $K_{IB}$ and $K_{BI}$
invertible square matrices. Thus the hypotheses of
Theorem~\ref{thm:single} hold for the same fixed input throughout the
potential class. Applying its reconstruction formula gives
\eqref{eq:positive-recovery}, with $v=u_I$, and proves uniqueness.
\end{proof}

\section{Computational examples and further remarks}\label{sec:examples}

We first make the rank enhancement in Section~\ref{sec:fractional} concrete on the simplest example $C_6$. 
Then we explain why the fractional power alone does not always imply higher coupling rank.

\subsection{Explicit matrices for the cycle with consecutive halves}
\begin{example}\label{ex:matrices}
Consider the six-vertex cycle.
Use the basis $(e_0,e_1,e_2,e_3,e_4,e_5)$, with
$I_C=(0,1,2)$ and $B_C=(3,4,5)$. Then
\begin{equation}\label{eq:C6-L}
 L_C=\begin{pmatrix}
 2&-1&0&0&0&-1\\
 -1&2&-1&0&0&0\\
 0&-1&2&-1&0&0\\
 0&0&-1&2&-1&0\\
 0&0&0&-1&2&-1\\
 -1&0&0&0&-1&2
 \end{pmatrix},\qquad
 E:=(L_C)_{I_CB_C}=\begin{pmatrix}0&0&-1\\0&0&0\\-1&0&0\end{pmatrix}.
\end{equation}

Let $\theta_k=2\pi k/6$ and
$\phi_k(j)=6^{-1/2}e^{i\theta_kj}$. A direct computation gives
\[
 \ip{\phi_k}{\phi_\ell}_{\Z_6}
 =\frac16\sum_{j=0}^5 e^{2\pi i(k-\ell)j/6}
 =\delta_{k\ell},\qquad 0\le k,\ell\le5.
\]
Thus these six vectors form an orthonormal basis. Note that
\[
 L_C\phi_k=(2-e^{i\theta_k}-e^{-i\theta_k})\phi_k
 =(2-2\cos\theta_k)\phi_k.
\]
The eigenvalues in the order $k=0,\ldots,5$ are
\[
 0,\ 1,\ 3,\ 4,\ 3,\ 1.
\]
For $\mu\ge0$, spectral calculus gives
\[(\mu\Id_6+L_C)^s= \sum_{k=0}^5(\mu+2-2\cos\theta_k)^s \phi_k\phi^*_k.\]
Pairing $k$ with $6-k$ gives
\begin{align}\label{eq:C6-Fourier}
 [(\mu\Id_6+L_C)^s]_{ij}
 &=\frac16\sum_{k=0}^5(\mu+2-2\cos\theta_k)^s
                 e^{i\theta_k(i-j)}\notag\\
 &=\frac16\left[\mu^s
  +2(\mu+1)^s\cos\frac{\pi(i-j)}3
  +2(\mu+3)^s\cos\frac{2\pi(i-j)}3
  +(\mu+4)^s(-1)^{i-j}\right].
\end{align}
We define four real coefficients
\begin{align}\label{eq:C6-coefficients}
 \alpha_\mu&=\frac{\mu^s+2(\mu+1)^s+2(\mu+3)^s+(\mu+4)^s}{6},\notag\\
 \beta_\mu&=\frac{\mu^s+(\mu+1)^s-(\mu+3)^s-(\mu+4)^s}{6},\notag\\
 \gamma_\mu&=\frac{\mu^s-(\mu+1)^s-(\mu+3)^s+(\mu+4)^s}{6},\\
 \delta_\mu&=\frac{\mu^s-2(\mu+1)^s+2(\mu+3)^s-(\mu+4)^s}{6}.\notag
\end{align}
Then the whole $6\times6$ shifted matrix is
\begin{equation}\label{eq:C6-Fmatrix}
 F_\mu=
 \begin{pmatrix}
 \alpha_\mu&\beta_\mu&\gamma_\mu&\delta_\mu&\gamma_\mu&\beta_\mu\\
 \beta_\mu&\alpha_\mu&\beta_\mu&\gamma_\mu&\delta_\mu&\gamma_\mu\\
 \gamma_\mu&\beta_\mu&\alpha_\mu&\beta_\mu&\gamma_\mu&\delta_\mu\\
 \delta_\mu&\gamma_\mu&\beta_\mu&\alpha_\mu&\beta_\mu&\gamma_\mu\\
 \gamma_\mu&\delta_\mu&\gamma_\mu&\beta_\mu&\alpha_\mu&\beta_\mu\\
 \beta_\mu&\gamma_\mu&\delta_\mu&\gamma_\mu&\beta_\mu&\alpha_\mu
 \end{pmatrix},\quad
 C_\mu=\begin{pmatrix}
 \delta_\mu&\gamma_\mu&\beta_\mu\\
 \gamma_\mu&\delta_\mu&\gamma_\mu\\
 \beta_\mu&\gamma_\mu&\delta_\mu
 \end{pmatrix}.
\end{equation}
In particular, $L_C^s=F_0$ and $(L_C^s)_{I_CB_C}=C_0$.
For clarity, its coefficients are
\begin{equation}\label{eq:C6-zero-coefficients}
 \begin{aligned}
 \alpha_0&=(2+2\cdot3^s+4^s)/6,&
 \beta_0&=(1-3^s-4^s)/6,\\
 \gamma_0&=(-1-3^s+4^s)/6,&
 \delta_0&=(-2+2\cdot3^s-4^s)/6.
 \end{aligned}
\end{equation}
Theorem~\ref{thm:reflected} implies $\rank C_0=3$, compared with
$\rank E=2$.

\end{example}

\subsection{The parity partition of the same cycle}

We remark that rank enhancement depends on the specific partition of a fixed graph. Keep the same cycle $C_6$ but choose the parity partition
\[
 I=(0,2,4),\qquad B=(1,3,5),
\]
with the displayed orderings. We obtain the relevant cross blocks
simply by selecting rows $0,2,4$ and columns $1,3,5$ from the full
matrices computed earlier. More explicitly, for the classical operator, selecting these entries in
\eqref{eq:C6-L} gives
\begin{equation}\label{eq:C6-parity-L}
 (L_C)_{IB}
 =\begin{pmatrix}
 -1&0&-1\\
 -1&-1&0\\
 0&-1&-1
 \end{pmatrix},
\end{equation}
which has rank three. For the fractional operator, taking $\mu=0$ and selecting the same rows and columns in \eqref{eq:C6-Fmatrix}, we obtain
\begin{equation}\label{eq:C6-parity-fractional}
 (L_C^s)_{IB}
 =\begin{pmatrix}
 \beta_0&\delta_0&\beta_0\\
 \beta_0&\beta_0&\delta_0\\
 \delta_0&\beta_0&\beta_0
 \end{pmatrix},
\end{equation}
which has rank three as well.

Thus the fractional power improves the coupling rank for the
consecutive-half partition, whereas the parity partition already has
full classical coupling rank.

\subsection{A complete graph}

We also remark that rank enhancement depends on the graph and its weights. For the unweighted
complete graph $K_N$, every vertex has degree $N-1$, its adjacency
matrix is $J_N-\Id_N$, where $J_N$ is the all-ones matrix. Then the Laplacian is
\[
 L=N\Id_N-J_N.
\]
Note that $P=N^{-1}J_N$ is the orthogonal projection onto constant
vectors: $P^2=P$, $P^*=P$, and $P\one=\one$.
The complement $\Id_N-P$ projects onto vectors with zero sum. Its eigenvalues are $0$ and $1$. It follows that
\[
 \ L^s= N^s\left(\Id_N-\frac1N J_N\right).
\]
Every cross block has rank one. Fractionalization does not improve the rank here.


\begin{thebibliography}{99}

\bibitem{BianchiEtAl}
Davide Bianchi, Marco Donatelli, Fabio Durastante, and Mariarosa Mazza,
\emph{Compatibility, embedding and regularization of non-local random walks
on graphs}, Journal of Mathematical Analysis and Applications
\textbf{511} (2022), no.~1, 126020.

\bibitem{BlastenIsozakiLassasLu}
Emilia Bl\aa sten, Hiroshi Isozaki, Matti Lassas, and Jinpeng Lu,
\emph{Inverse problems for discrete heat equations and random walks for a
class of graphs}, SIAM Journal on Discrete Mathematics
\textbf{37} (2023), no.~2, 831--863.

\bibitem{BoyerGarzellaGuevaraVasquez}
Justin Boyer, Jack Garzella, and Fernando Guevara Vasquez,
\emph{On the solvability of the discrete conductivity and Schr\"odinger
inverse problems}, SIAM Journal on Applied Mathematics
\textbf{76} (2016), no.~3, 1053--1075.

\bibitem{Chung}
Fan Chung, \emph{Spectral Graph Theory}, CBMS Regional Conference Series
in Mathematics, vol.~92, American Mathematical Society, Providence, RI,
1997.

\bibitem{CoviLassas}
Giovanni Covi and Matti Lassas,
\emph{An inverse problem for fractional random walks on finite graphs},
Inverse Problems \textbf{42} (2026), no.~4, 045010.

\bibitem{Ellis}
Robert Ellis, \emph{Discrete Green's functions for products of regular
graphs}, arXiv:math/0309080, 2003.

\bibitem{FeizmohammadiGhoshKrupchykUhlmann}
Ali Feizmohammadi, Tuhin Ghosh, Katya Krupchyk, and Gunther Uhlmann,
\emph{Fractional anisotropic Calder\'on problem on closed Riemannian
manifolds}, Journal of Differential Geometry \textbf{131} (2025),
no.~2, 401--414.

\bibitem{FeizmohammadiKrupchykUhlmann}
Ali Feizmohammadi, Katya Krupchyk, and Gunther Uhlmann,
\emph{Calder\'on problem for fractional Schr\"odinger operators on closed
Riemannian manifolds}, Inventiones Mathematicae \textbf{246} (2026),
13--66.

\bibitem{FernandezRoncalRuland}
Aingeru Fern\'andez-Bertolin, Luz Roncal, and Angkana R\"uland,
\emph{On (global) unique continuation properties of the fractional discrete
Laplacian}, Journal of Functional Analysis \textbf{286} (2024), no.~9,
110375.

\bibitem{GhoshRulandSaloUhlmann}
Tuhin Ghosh, Angkana R\"uland, Mikko Salo, and Gunther Uhlmann,
\emph{Uniqueness and reconstruction for the fractional Calder\'on problem
with a single measurement}, Journal of Functional Analysis
\textbf{279} (2020), no.~1, 108505.

\bibitem{GhoshSaloUhlmann}
Tuhin Ghosh, Mikko Salo, and Gunther Uhlmann,
\emph{The Calder\'on problem for the fractional Schr\"odinger equation},
Analysis \& PDE \textbf{13} (2020), no.~2, 455--475.

\bibitem{Kwasnicki}
Mateusz Kwa\'snicki,
\emph{Ten equivalent definitions of the fractional Laplace operator},
Fractional Calculus and Applied Analysis \textbf{20} (2017), no.~1,
7--51.

\bibitem{RulandSalo}
Angkana R\"uland and Mikko Salo,
\emph{The fractional Calder\'on problem: low regularity and stability},
Nonlinear Analysis \textbf{193} (2020), 111529.

\bibitem{Serre}
Jean-Pierre Serre, \emph{Linear Representations of Finite Groups},
Graduate Texts in Mathematics, vol.~42, Springer-Verlag, New York, 1977.

\bibitem{SylvesterUhlmann}
John Sylvester and Gunther Uhlmann,
\emph{A global uniqueness theorem for an inverse boundary value problem},
Annals of Mathematics (2) \textbf{125} (1987), no.~1, 153--169.

\end{thebibliography}
\end{document}